\documentclass[reqno]{amsart}
\usepackage{amssymb,amsmath,amsfonts,amsthm,latexsym}
\usepackage[english]{babel}
\newtheorem{thm}{Theorem}[section]
\newtheorem{cor}[thm]{Corollary}

\newtheorem{prop}[thm]{Proposition}
\newtheorem{defn}[thm]{Definition}
\newtheorem{rem}[thm]{Remark}

\DeclareMathOperator{\Ann}{Ann}
\DeclareMathOperator{\cent}{Center}
\DeclareMathOperator{\gr}{gr}

\begin{document}
\title[Classification of solvable Leibniz algebras] {Classification of solvable Leibniz algebras with naturally graded filiform nilradical}
\author{J. M. Casas, M. Ladra, B. A. Omirov and I. A. Karimjanov}
\address{[J.M. Casas] Department of Applied Mathematics I, University of Vigo, E. E. Forestal, 36005 Pontevedra, Spain}
 \email{jmcasas@uvigo.es}
\address{[M. Ladra] Department of Algebra, University of Santiago de Compostela, 15782 Santiago de Compostela, Spain}
 \email{manuel.ladra@usc.es}
\address{[B.A. Omirov --- I. A. Karimdjanov] Institute of Mathematics and Information Technologies  of Academy of Uzbekistan, 29, Do'rmon yo'li street., 100125, Tashkent (Uzbekistan)} \email{omirovb@mail.ru --- iqboli@gmail.com}



\begin{abstract} In this paper we show that the method for  describing solvable Lie algebras with given nilradical by
 means of non-nilpotent outer derivations of the  nilradical is also applicable to the case of Leibniz algebras. Using
 this method we extend the classification of solvable Lie algebras with naturally graded filiform Lie algebra to the
 case of Leibniz algebras. Namely, the classification of solvable Leibniz algebras whose nilradical is a  naturally
 graded filiform Leibniz algebra is obtained.
\end{abstract}

\maketitle
\section{Introduction}

Leibniz algebras were introduced at the beginning of the 90s of the past century by J.-L. Loday
in \cite{Lod}. They are a ``non-commutative'' generalization of Lie algebras. Leibniz algebras inherit an important property of Lie algebras which is that the right multiplication operator on an element of a Leibniz algebra is a derivation. Active investigations on Leibniz algebras  theory  show that many results of the theory of Lie algebras can be  extended to Leibniz algebras. Of course, distinctive properties of non-Lie Leibniz algebras have also  been studied \cite{AyOm2,Bar}.

In fact, for a Leibniz algebra we have the corresponding Lie algebra, which is the quotient algebra by the two-sided ideal $I$ generated by the square  elements of a Leibniz algebra. Notice that this ideal is the minimal one such that the quotient algebra is a Lie algebra and in the  case of non-Lie Leibniz algebras it is always non trivial (moreover, it is abelian).

From the theory of Lie algebras it is well known that the study of finite dimensional Lie algebras was reduced to the nilpotent ones \cite{Jac,Mal}. In Leibniz algebras case we have an analogue of Levi's theorem \cite{Bar}. Namely, the decomposition of a Leibniz algebra into a semidirect sum of its solvable radical and a semisimple Lie algebra is obtained. The semisimple part can be described from simple Lie ideals and therefore, the main problem is to study the solvable radical, i.e. in a similar way as in the case of Lie algebras, the description of Leibniz algebras is reduced to  the description of the solvable ones. The analysis of works devoted to the study of solvable Lie algebras (for example \cite{AnCaGa1,AnCaGa2,NdWi,TrWi,WaLiDe}, where solvable Lie algebras with various types of nilradical were studied, such as naturally graded filiform and quasi-filiform algebras, abelian, triangular, etc.) shows that we can also apply similar methods to solvable Leibniz algebras with a given nilradical. Some results of Lie algebras theory generalized to Leibniz algebras \cite{AyOm1} allow to apply the technique of description of solvable extensions of nilpotent Lie algebras to the case of Leibniz algebras.

The aim of the present paper is to classify solvable Leibniz algebras with naturally graded filiform nilradical. Thanks to the works \cite{AyOm2} and \cite{Ver}, we already have  the classification of naturally graded filiform Leibniz algebras.

In order to achieve our goal we organize the paper as follows: in Section \ref{S:Prel} we give some necessary notions and preliminary results about Leibniz algebras and solvable Lie algebras with naturally graded filiform radical.  Section \ref{S:nil_lie} is devoted to the classification of solvable Leibniz algebras whose nilradical is a naturally graded filiform Lie algebra and in  Section \ref{S:nil_leib} we describe, up to isomorphisms, solvable Leibniz algebras whose nilradical is a naturally graded filiform non-Lie Leibniz algebra.

Throughout the paper vector spaces and algebras are  finite-dimensional over the field of the complex numbers. Moreover, in the table of multiplication  of an algebra the omitted products are assumed to be zero and, if it is not noticed, we shall consider non-nilpotent solvable algebras.

\section{Preliminaries} \label{S:Prel}

In this section we give necessary definitions and preliminary results.

\begin{defn} An algebra $(L,[-,-])$ over a field $F$ is called a Leibniz algebra if for any $x,y,z\in L$ the so-called Leibniz identity
\[ \big[x,[y,z]\big]=\big[[x,y],z\big] - \big[[x,z],y\big] \]  holds.
\end{defn}

From the Leibniz identity we conclude that the elements $[x,x], [x,y]+[y,x]$, for any $x, y \in L$, lie in  $\Ann_r(L) =\{x \in L \mid [y,x] = 0,\ \text{for \ all}\ y \in L \}$, the \emph{right annihilator}  of the Leibniz algebra $L$. Moreover, we also get that $\Ann_r(L)$ is a two-sided ideal of $L$.

The two-sided ideal  $\cent(L)=\{x\in L  \mid   [x,y]=0=[y,x], \ \text{for \ all}\ y \in L \}$ is said to be the \emph{center} of $L$.

\begin{defn} A linear map $d \colon L \rightarrow L$ of a Leibniz algebra $(L,[-,-])$ is said to be a derivation if for all $x, y \in L$, the following  condition holds: \[d([x,y])=[d(x),y] + [x, d(y)] \,.\]
\end{defn}

For a given element $x$ of a Leibniz algebra $L$, the right multiplication operators $\mathcal{R}_x \colon L \to L, \mathcal{R}_x(y)=[y,x],  y \in L$, are derivations. This kind of derivations are said to be \emph{inner derivations}. Any Leibniz algebra $L$ has associated the algebra  of right multiplications $\mathcal{R}(L)=\{ \mathcal{R}_x  \mid   x\in L\}$. $\mathcal{R}(L)$ is endowed with a structure  of Lie algebra by means of the bracket $[\mathcal{R}_x,\mathcal{R}_y]=\mathcal{R}_x \mathcal{R}_y - \mathcal{R}_y \mathcal{R}_x = \mathcal{R}_{[y,x]}$. Moreover,
there is an antisymmetric isomorphism between   $\mathcal{R}(L)$ and the quotient algebra $L/\Ann_r(L)$.

\begin{defn}For a given Leibniz algebra $(L,[-,-])$ the sequences of two-sided ideals defined recursively as follows:
\[L^1=L, \ L^{k+1}=[L^k,L],  \ k \geq 1, \qquad \qquad
L^{[1]}=L, \ L^{[s+1]}=[L^{[s]},L^{[s]}], \ s \geq 1.
\]
are said to be the lower central and the derived series of $L$, respectively.
\end{defn}

\begin{defn} A Leibniz algebra $L$ is said to be
nilpotent (respectively, solvable), if there exists $n\in\mathbb N$ ($m\in\mathbb N$) such that $L^{n}=0$ (respectively, $L^{[m]}=0$).
The minimal number $n$ (respectively, $m$) with such property is said to be the index of
nilpotency (respectively, of solvability) of the algebra $L$.
\end{defn}

Evidently, the index of nilpotency of an $n$-dimensional algebra is not greater than $n+1$.

\begin{defn} An $n$-dimensional Leibniz algebra $L$ is said to be null-filiform if $\dim L^i=n+1-i, \ 1\leq i \leq n+1$.
\end{defn}
Evidently, null-filiform Leibniz algebras have maximal index of nilpotency.

\begin{thm}[\cite{AyOm2}]   An arbitrary $n$-dimensional null-filiform Leibniz algebra is isomorphic to the algebra
\[NF_n: \quad [e_i, e_1]=e_{i+1}, \quad 1 \leq i \leq n-1,\]
where $\{e_1, e_2, \dots, e_n\}$ is a basis of the algebra $NF_n$.
\end{thm}

Actually, a nilpotent Leibniz algebra is null-filiform if and only if it is one-generated algebra. Notice that this notion has no sense in Lie algebras case, because they are at least two-generated.

\begin{defn} An $n$-dimensional Leibniz algebra $L$ is said to be filiform if
 $\dim L^i=n-i$, for $2\leq i \leq n$.
\end{defn}

Now let us define a naturally graduation for a filiform Leibniz algebra.

\begin{defn} Given a filiform Leibniz algebra $L$, put
$L_i=L^i/L^{i+1}, \ 1 \leq i\leq n-1$, and $\gr(L) = L_1 \oplus
L_2\oplus\dots \oplus L_{n-1}$. Then $[L_i,L_j]\subseteq L_{i+j}$ and we
obtain the graded algebra $\gr(L)$. If $\gr(L)$ and $L$ are isomorphic, then
we say that an algebra $L$ is naturally graded.
\end{defn}

Thanks to \cite{Ver} it is well known that there are two types of naturally graded filiform Lie algebras.
In fact, the second type will appear only in the case when the dimension of the algebra is even.

\begin{thm}[\cite{Ver}] \label{thm2.8} Any complex naturally graded filiform Lie algebra is isomorphic to one of the following non isomorphic algebras:
\[n_{n,1}: [e_i,e_1]=-[e_1,e_i]=e_{i+1}, \quad 2\leq i \leq n-1.\]

\[Q_{2n}:\left\{\begin{aligned}
{}[e_i,e_1] & =  -[e_1,e_i]=e_{i+1},&& 2\leq i \leq 2n-2,\\
[e_i,e_{2n+1-i}]  & =  -[e_{2n+1-i},e_i]=(-1)^i\,e_{2n},&& 2\leq i \leq n.
\end{aligned}\right.\]
\end{thm}

In the following theorem we recall the classification of the naturally graded filiform
non-Lie Leibniz algebras given in \cite{AyOm2}.

\begin{thm}[\cite{AyOm2}] \label{thm28} Any complex $n$-dimensional naturally graded filiform
non-Lie Leibniz algebra is isomorphic to one of the following non
isomorphic algebras:
\[F_n^1=\left\{\begin{array}{llll}
[e_1,e_1] & = &  e_{3},&  \\[1mm]
[e_i,e_1] & =  & e_{i+1}, & \  2\leq i \leq {n-1},
\end{array} \right.\quad \quad F_n^2=\left\{\begin{array}{llll}
[e_1,e_1]& = & e_{3}, &  \\[1mm]
[e_i,e_1]& = & e_{i+1}, & \  3\leq i \leq {n-1}.
\end{array} \right.\]
\end{thm}

\begin{defn} The maximal nilpotent ideal of a Leibniz algebra is said to be the nilradical of the  algebra.
\end{defn}

Notice that the nilradical is not the radical in the sense of Kurosh, because the quotient Leibniz algebra by its nilradical may contain a nilpotent ideal (see \cite{Jac}).

All solvable Lie algebras whose nilradical is the naturally graded filiform Lie algebra $n_{n,1}$ are classified in \cite{SnWi}. Further  solvable Lie algebras whose nilradical is the naturally graded filiform Lie algebra $Q_{2n}$ are classified in \cite{AnCaGa1x}.

Using the above classifications, we shall give the classification of
solvable non-Lie Leibniz algebras whose nilradical is a naturally graded filiform Lie algebra.

It is proved that the dimension of a solvable Lie algebra whose nilradical is isomorphic
to an $n$-dimensional naturally graded filiform Lie algebra is not greater than $n+2$. Below, we present their classification.

In order to agree with the tables of multiplications of algebras in Theorems \ref{thm2.8} and \ref{thm28}, we make the following change of basis in the classification of \cite{SnWi}:
\[e'_i = e_{n+1-i}, \quad 1 \leq i \leq n, \qquad \qquad x=-f.\]
We also use different notation to denote the algebras that appear in \cite{SnWi}. That way the results would be:

\begin{thm}[\cite{SnWi}] \label{thm2.11} There are three types of solvable Lie algebras of dimension $n+1$ with
nilradical isomorphic to $n_{n,1}$, for any $n \geq 4$. The
isomorphism classes in the basis $\{e_1, \dots,e_n,x \}$ are
represented by the following algebras:
\[S_{n+1} (\alpha, \beta):\left\{\begin{aligned}
{} [e_i,e_1]&=-[e_1,e_i]=e_{i+1},&& 2 \leq i \leq n-1,\\
[e_i,x]&=-[x,e_i]=\big((i-2)\alpha +\beta \big)\, e_i, &&2\leq i\leq n,\\
[e_1,x]&=-[x,e_1]= \alpha e_1\,.&&
\end{aligned}\right.\]
The mutually non-isomorphic algebras of this type are $S_{n+1,1}(\beta) = S_{n+1} (1, \beta)$ and 
$S_{n+1,2} = S_{n+1} (0, 1)$.
\[ S_{n+1,3} : \left\{\begin{aligned}
{} [e_i,e_1]& =-[e_1,e_i]=e_{i+1},&& 2 \leq i \leq n-1,\\
[e_i,x]& =-[x,e_i]=(i-1)\, e_i, && 2\leq i\leq n,\\
[e_1,x]& =-[x,e_1]=e_1 + e_2\,.&&
\end{aligned}\right.\]
\[S_{n+1,4}(a_3, a_4, \dots, a_{n-1}) :
 \left\{\begin{aligned}
{}[e_i,e_1]&=-[e_1,e_i]=e_{i+1},&& 2 \leq i \leq n-1,\\
[e_i,x]&=-[x,e_i]=e_i+\sum\limits_{l=i+2}^{n}a_{l+1-i}\, e_l, &&2\leq i\leq n,
\end{aligned}\right.\]
 where the first non-vanishing parameter $\{a_3, \dots, a_{n-1}\}$
can be assumed to be equal to 1.\end{thm}

\begin{thm}[\cite{SnWi}] \label{thm2.12} There exists only one class of solvable Lie algebras of dimension $n+2$ with
nilradical $n_{n,1}$. It is represented by a basis $\{e_1, e_2,
\dots, e_n, x, y\}$ and the Lie brackets are
\[S_{n+2}:\left\{\begin{aligned}
{}[e_i,e_1]&=-[e_1,e_i]=e_{i+1},&&  2 \leq i \leq n-1,\\
[e_i,x]&=-[x,e_i]=(i-2)\, e_i, && 2\leq i\leq n,\\
[e_1,x]&=-[x,e_1]=e_1,&&\\
 [e_i,y]&=-[y,e_i]=e_i, && 2\leq i\leq n.
\end{aligned}\right.\]
\end{thm}

Now we recall the classification given in \cite{AnCaGa1x} after the
following change of basis: \[e_1'=-e_1, \qquad x'=-Y_1, \qquad y'=-Y_2 \,.\]

\begin{prop}[\cite{AnCaGa1x}] \label{prop2.13} Any solvable Lie algebra of dimension 2n+1 with nilradical isomorphic to $Q_{2n}$
is isomorphic to one of the following algebras:
\[ Q_{2n+1,1}(\alpha):  \left\{\begin{aligned}
{}[e_i,e_1]& =-[e_1,e_i]=e_{i+1},&&  2\leq i \leq 2n-2,\\
[e_i,e_{2n+1-i}]& =-[e_{2n+1-i},e_i]=(-1)^i\, e_{2n},&& 2\leq i \leq n,\\
 [e_1,x]&=-[x,e_1]=e_1,&&\\
 [e_i,x]&=-[x,e_i]=(i-2+\alpha)\,e_i, &&2\leq i\leq 2n-1,\\
  [e_{2n},x]&=-[x,e_{2n}]=(2n-3-2\alpha)\,e_{2n}\,.&&
\end{aligned}\right.\]

\[Q_{2n+1,2}:  \left\{\begin{aligned}
{}[e_i,e_1]& =-[e_1,e_i]=e_{i+1},&& 2\leq i \leq 2n-2,\\
[e_i,e_{2n+1-i}]&=-[e_{2n+1-i},e_i]=(-1)^i\, e_{2n},&& 2\leq i \leq n,\\
 [e_1,x]& =-[x,e_1]=e_1 + \varepsilon\, e_{2n},&& \varepsilon = 0, 1, \\
  [e_i,x]&=-[x,e_i]=(i-n)\, e_i, && 2\leq i\leq 2n-1,\\
[e_{2n},x]&=-[x, e_{2n}]= e_{2n}\,.&&
\end{aligned}\right.\]

\[Q_{2n+1,3}(\alpha):  \left\{\begin{aligned}
{}[e_i,e_1]& =-[e_1,e_i]=e_{i+1},&& 2\leq i \leq 2n-2,\\
[e_i,e_{2n+1-i}] & =-[e_{2n+1-i},e_i]=(-1)^i\, e_{2n},&& 2\leq i \leq n,\\
 [e_{2+i},x] & =-[x,e_{2+i}]=e_{2+i} + \sum\limits_{k=2}^{\lfloor \frac {2n-3-i} {2} \rfloor}
  \alpha^{2k+1}\, e_{2k+1+i}, && 0\leq i \leq 2n-6,\\
 [e_{2n-i},x]& =-[x,e_{2n-i}]=e_{2n-i}, && i = 1,2,3,\\
  [e_{2n},x]& =-[x,e_{2n}]= 2\, e_{2n}\,.&&
\end{aligned}\right. \]
\end{prop}

\begin{prop}[\cite{AnCaGa1x}] \label{prop2.14} For any $n \geq 3$ there is only one $(2n+2)$-dimensional solvable Lie algebra
having a nilradical isomorphic to $Q_{2n}$:
\[\left\{\begin{aligned}
{}[e_i,e_1]&=-[e_1,e_i]=e_{i+1},&& 2\leq i \leq 2n-2,\\
[e_i,e_{2n+1-i}]&=-[e_{2n+1-i},e_i]=(-1)^i\, e_{2n},&& 2\leq i \leq n,\\
 [e_i,x]&=-[x,e_i]=i\, e_i, && 1\leq i \leq 2n-1,\\
[e_{2n},x]&=-[x,e_{2n}]=(2n+1)\, e_{2n}, &&\\
 [e_i,y]&=-[y,e_i]=e_i, && 1\leq i \leq 2n-1,\\
 [e_{2n},y]&=-[y,e_{2n}]=2\, e_{2n} \,.  &&
\end{aligned}\right.\]
\end{prop}

Let $R$ be a solvable Leibniz algebra with nilradical $N$. We
denote by $Q$ the complementary vector space  of the nilradical
$N$ to the algebra $R$. Let us consider the restrictions to $N$ of
the right multiplication operator on an element $x \in Q$ (denoted
by $\mathcal{R}_{{x |}_{N}}$). If the operator $\mathcal{R}_{{x |}_{N}}$ is nilpotent,
then  we assert that the subspace $\langle x+N \rangle$ is a
nilpotent ideal of the algebra $R$. Indeed, since for a solvable
Leibniz algebra $R$ we get the inclusion $R^2\subseteq N$
\cite{AyOm1}, and hence the subspace $\langle x+N \rangle $ is an
ideal. The nilpotency of this ideal follows from the Engel's
theorem for Leibniz algebras \cite{AyOm1}. Therefore, we have a
nilpotent ideal which strictly contains the nilradical, which is
in contradiction with the maximality of $N$. Thus, we obtain that
for any $x \in Q$, the operator $\mathcal{R}_{{x |}_{N}}$ is a non-nilpotent
 derivation of $N$.

 Let $\{x_1, \dots, x_m\}$ be a basis of
$Q$, then for any scalars $\{\alpha_1, \dots, \alpha_m\}\in
\mathbb{C}\setminus\{0\}$, the matrix $\alpha_1\mathcal{R}_{{x_1
|}_{N}}+\dots+\alpha_m\mathcal{R}_{{x_m|}_{N}}$ is not nilpotent, which
means that the elements $\{x_1, \dots, x_m\}$ are nil-independent
\cite{Mub}. Therefore, we have that the dimension of $Q$ is bounded by
the maximal number of nil-independent derivations of the
nilradical $N$. Moreover, similarly to the case of Lie algebras,
for a solvable Leibniz algebra $R$ the inequality $\dim N \geq
\frac{\dim R}{2}$ holds.

\section{solvable Leibniz algebras whose nilradical is a Lie algebra} \label{S:nil_lie}

It is not difficult to see that if $R$ is a solvable non-Lie
Leibniz algebra with nilradical isomorphic to the algebras
$n_{n,1}$ or $Q_{2n}$, then the dimension of $R$  is also not
greater than $n+2$ and $2n+2$, respectively.

Let $n_{n,1}$ or $Q_{2n}$ be the nilradical of a solvable Leibniz
algebra $R$. Since the ideal $I=\langle \{[x,x] \mid x \in R \}
\rangle$ is contained in $\Ann_r(R)$, then $I$ is abelian, hence it
is contained in the  nilradical. Taking into account the multiplication
in $n_{n,1}$ (respectively $Q_{2n}$) we conclude that $I=\langle
\{e_n\} \rangle$.

Having in mind that an $(n+1)$-dimensional algebra $R$ is
solvable, then the quotient algebra $R/I$ is also a solvable Lie
algebra with nilradical $n_{n,1}$ (whose lists of tables of
multiplication are given  in Theorems \ref{thm2.11} and
\ref{thm2.12}).

\textbf{Case $n_{n,1}$.} Let us assume that $R$ has dimension
$n+1$, then the table of multiplication in $R$ will be equal to
the table of multiplication of $S_{n+1,i}$, ($i=1, 2, 3, 4$),
except the following products:
 \begin{align*}
 [e_1,x] &= \alpha_1e_1 +\gamma_4e_n,   & [e_2, x] & = \beta_1e_2 +\gamma_5e_n, &\\
 [x, e_1]&= -\alpha_1e_1 +\gamma_1e_n,  & [x, e_2] & = -\beta_1e_2 +\gamma_2e_n, & [x, x] = \gamma_3e_n,
 \end{align*}
where $\{\gamma_1+\gamma_4, \gamma_2+\gamma_5, \gamma_3\} \neq \{0, 0, 0\}$.

Note that taking the change of basis
\[e'_1 = \alpha_1\, e_1+\gamma_4\, e_n, \qquad e'_2=\beta_1\, e_2 +\gamma_5\, e_n\]
we can assume
that $\gamma_4=\gamma_5=0$, i.e., $[e_1,x] = \alpha\, e_1$ and
$[e_2, x] = \beta_1\, e_2$.

It is not difficult to see that, for the omitted products, the
antisymmetric identity  holds, i.e.

\[\left\{\begin{aligned}
{}[e_i,e_1] & = - [e_1, e_i] = e_{i+1},&& 2 \leq i \leq n-1,\\
[e_i,x]&=-[x, e_i], &&3\leq i\leq n\,.
\end{aligned} \right.\]

We have $[e_n, x] =0$ because  $0=[x, e_n]=-[e_n,x]$.

Consider
\[0=[x,e_n]=\big[x,[e_{n-1},e_1]\big]=\big[[x,e_{n-1}],e_1\big]-\big[[x,e_1],e_{n-1}\big]=-(n-2+\beta)\, e_n \,.\]

In the list of Theorem \ref{thm2.11} only the algebra
$S_{n+1,1}(\beta)$ is representative of the class for which
the equality $[e_n, x] =0$  holds. This class is defined by $\beta
= 2-n$.

Therefore, in the case of $\dim R = n+1$ whose nilradical is
$n_{n,1}$, we have the following family:

\[R_{n+1,1} (\gamma_1, \gamma_2, \gamma_3):\left\{\begin{aligned}
{}[e_i,e_1] & = - [e_1, e_i] = e_{i+1},&& 2 \leq i \leq n-1, \\
 [e_1, x] &= e_1,&& \\
 [x,e_1] &= -e_1+\gamma_1\, e_n,&&\\
 [e_2, x]&= (2-n)\, e_2,&&\\
 [x,e_2]&= (n-2)\, e_2 +\gamma_2\, e_n,&& \\
[e_i,x]&=-[x, e_i] = (i-n)\, e_i, &&3\leq i\leq n-1, \\
 [x, x] &= \gamma_3\, e_n, &&
\end{aligned}\right. \]
where $(\gamma_1, \gamma_2, \gamma_3)\neq (0,0, 0)$.

Applying a similar argument and the table of multiplication
of the algebra in Theorem \ref{thm2.12} we conclude that solvable non-Lie Leibniz algebras of dimension $n+2$ with
nilradical $n_{n,1}$ do not exist.

\begin{thm} \label{thm4.1} Any $(n+1)$-dimensional solvable Leibniz algebra with nilradical $n_{n,1}$ is isomorphic to one of the following pairwise non isomorphic algebras:
\[R_{n+1,1}(0, 0, 1), \quad \ R_{n+1,1}(0, 1, 0), \quad \ R_{n+1,1}(1, 1, 0), \quad \ R_{n+1,1}(1, 0, 0).\]
\end{thm}
\begin{proof}
We consider the general change of basis in the family
$R_{n+1,1}(\gamma_1, \gamma_2, \gamma_3)$:
\[e'_1 = \sum_{i=1}^{n}A_i\, e_i, \qquad e'_2 = \sum_{i=1}^{n}B_i\, e_i, \qquad x' = D\, x + \sum_{i=1}^{n}C_i\, e_i,\]
where $(A_1B_2 - B_1A_2)D \neq 0$.

Using $[e'_{i}, e'_1]=e_{i+1}', \ 2 \leq i \leq n-1$, the table of
multiplication of $R_{n+1,1} (\gamma_1, \gamma_2, \gamma_3)$ and
an induction, we obtain
\[e_i'=A_1^{i-3}\,\sum_{j=i}^{n}(A_1B_{j+2-i}-B_1A_{j+2-i})\, e_j, \ 3\leq i\leq n.\]

From the equalities
\[0=[e'_3, e'_2] =B_1\sum_{j=4}^{n}
(A_1B_{j-2}-B_1A_{j-2})\, e_j\] we have $B_1 =0$.

 Consider the multiplications

\begin{align*}
[e'_1, x'] &= A_1D\, e_1 - D\sum_{i=2}^{n-1}A_i(n-i)\, e_i+ \sum_{i=3}^{n}(A_{i-1}C_1-A_1C_{i-1})\, e_i \\
& {} = A_1De_1 - A_2D(n-2)\, e_2+ \sum_{i=3}^{n-1}\big(A_{i-1}C_1-A_1C_{i-1}-(n-i)A_iD\big)\, e_i \\
 & {}+ {} (A_{n-1}C_1-A_1C_{n-1})\, e_n.
\end{align*}
On the other hand \[[e'_1,x'] = e'_1 = \sum_{i=1}^{n}A_i\, e_i.\]

Comparing the coefficients of the basic elements we derive:
\[D=1,\quad A_2=0, \quad A_{i+1}
=\frac {A_1C_i-A_iC_1} {i-n-1}, \  \ 2 \leq i \leq n-2, \quad A_n =
A_1 C_{n-1} - A_{n-1} C_n \,.\]

From the equalities
\begin{align*}
-(n-2)\sum_{i=2}^{n}B_ie_i & {}= {} -(n-2) e'_2=[e'_2, x'] = \big[\sum_{i=2}^{n}B_ie_i, x + \sum_{i=1}^{n}C_ie_i\big] \\
& {}= {} -\sum_{i=2}^{n-1}B_i(n-i)e_i+
C_1\sum_{i=3}^{n}B_{i-1}e_i \\
& {}= {} -B_2(n-2)e_2 +
\sum_{i=3}^{n-1}\big(B_{i-1}C_1-B_i(n-i)\big)e_i+B_{n-1}C_1e_n\,,
\end{align*}
we deduce the following  restrictions:
\[B_i = (-1)^i\frac {B_2C_1^{i-2}} {(i-2)!},  \qquad  3 \leq i \leq n \,.\]

In an analogous way, comparing coefficients at the basic element
$e_n$ in the equalities, we obtain:
\[\gamma'_3
A_1^{n-2}B_2\, e_n=\gamma'_3\, e'_n =[x', x'] = (\gamma_3 +C_1
\gamma_1 +C_2\gamma_2)\, e_n\] and so
\[\gamma'_3 = \frac {\gamma_3 +C_1 \gamma_1 +C_2\gamma_2} {A_1^{n-2}B_2}.\]

With a similar argument,
\[-e'_1 + A_1^{n-2}B_2\gamma_1'\, e_n=-e'_1 + \gamma_1'\, e'_n =[x',e'_1] = -e'_1 + A_1\gamma_1\, e_n.\]
and
\[-(n-2)\, e'_2 + A_1^{n-2}B_2\gamma_2'\, e_n=(n-2)\, e'_2 + \gamma_2'\, e'_n=[x',e'_2]=(n-2)\, e'_2 +B_2\gamma_2\, e_n\]
we obtain
\[\gamma_1' = \frac {\gamma_1} {A_1^{n-3}B_2} \qquad \text{and} \qquad \gamma_2' = \frac {\gamma_2} {A_1^{n-2}}.\]

Now we shall consider the possible cases of the parameters
$\{\gamma_1, \gamma_2,  \gamma_3\}$.

\textbf{Case 1.} Let $\gamma_1 = 0$. Then $\gamma_1' = 0$.

If $\gamma_2 = 0$, then $\gamma_2' = 0$ and $\gamma'_3 = \frac
{\gamma_3} {A_1^{n-2}B_2} \neq 0$. Putting $B_2 = \frac {\gamma_3}
{A_1^{n-2}}$, then we have that $\gamma'_3=1$, so the algebra is
$R_{n+1,1}(0, 0, 1)$.

If $\gamma_2 \neq 0$, then putting $A_1 = \sqrt[\uproot{3} n-2]{\gamma_2}$
and $C_2 = -\frac {\gamma_3} {\gamma_2}$, we get $\gamma_2' = 1$
and $\gamma'_3 = 0$, i.e. we obtain the algebra $R_{n+1,1}(0, 1,
0)$.

\textbf{Case 2.} Let  $\gamma_1 \neq 0$. Then putting $B_2 =
\frac {\gamma_1} {A_1^{n-3}}$ and $C_1 = -\frac {\gamma_3+C_2
\gamma_2} {\gamma_1}$, we have:
\[\gamma_1' = 1, \qquad \gamma_2' = \dfrac {\gamma_2} {A_1^{n-2}}, \qquad \gamma_3' =0.\]

If $\gamma_2 \neq 0$, then putting $A_1 = \sqrt[\uproot{3} n-2]{\gamma_2}$ we
have that $\gamma_2' = 1$, so we obtain the algebra $R_{n+1,1}(1,
1, 0)$.

If $\gamma_2 = 0$, then we get the algebra $R_{n+1,1}(1, 0, 0)$.
\end{proof}

\textbf{Case $Q_{2n}$.} Similarly as above, from Propositions
\ref{prop2.13} and \ref{prop2.14}, we conclude that solvable
non-Lie Leibniz algebras with nilradical $Q_{2n}$ exist only in
the case of $\dim R=2n+1$ and they are isomorphic to $Q_{2n+1,1}(\alpha)$ for
$\alpha = \frac {2n-3} 2$. Thus, we have
\[R_{2n+1,1}:\left\{\begin{aligned}
{}[e_i,e_1] & = - [e_1, e_i] = e_{i+1},&& 2\leq i \leq 2n-2,\\
[e_i,e_{2n+1-i}] & = - [e_{2n+1-i}, e_i] = (-1)^i\, e_{2n},&&  2\leq i \leq n,\\
 [e_1, x]& = e_1, && \\
 [x,e_1]& = -e_1+ \gamma_1\, e_n,&&\\
  [e_2, x]& = \frac {2n-3}{2}\, e_2, &&\\
  [x, e_2]& = -\frac{2n-3}{2}\,  e_2 +\gamma_2\, e_n, &&\\
  [e_i,x]& = - [x,e_i] = \frac{2n+2i-7} 2\, e_i, &&3\leq i\leq 2n-1,\\
[x, x]&  = \gamma_3\, e_n,&&
\end{aligned}\right.\]
where $(\gamma_1, \gamma_2, \gamma_3)\neq (0,0, 0)$.

\begin{thm}  Any $(2n+1)$-dimensional solvable Leibniz algebra with nilradical $Q_{2n}$ is isomorphic to one of the following pairwise non isomorphic algebras:
\[R_{2n+1,1}(0, 0, 1), \quad \ R_{2n+1,1}(0, 1, 0), \quad \ R_{2n+1,1}(1, 1, 0), \quad \ R_{2n+1,1}(1, 0, 0).\]
\end{thm}

\begin{proof} The proof is carried out by applying similar arguments as in the proof of Theorem \ref{thm4.1}
\end{proof}

\section{solvable Leibniz algebras whose nilradical is a non-Lie Leibniz algebra} \label{S:nil_leib}

In the following proposition we describe derivations of the algebra $F_n^1$.
\begin{prop} \label{prop31} Any derivation of the algebra $F_n^1$ has the following matrix form:
 \[
\begin{pmatrix}
\alpha_1& \alpha_2&\alpha_3&\alpha_4&\dots&\alpha_{n-1}&\alpha_n\\
0& \alpha_1+\alpha_2&\alpha_3&\alpha_4&\dots&\alpha_{n-1}& \beta\\
0& 0&2\alpha_1+\alpha_2&\alpha_3&\dots&\alpha_{n-2}& \alpha_{n-1}\\
0& 0&0&3\alpha_1+\alpha_2&\dots&\alpha_{n-3}&\alpha_{n-2}\\
\vdots&\vdots&\vdots&\vdots&\dots&\vdots&\vdots\\
0&0&0&0&\dots&0&(n-1)\alpha_1+\alpha_2
\end{pmatrix}.
\]
\end{prop}
\begin{proof} Let $d$ be a derivation of the algebra. We set
\[d(e_1)=\sum_{i=1}^{n}\alpha_i\, e_i, \qquad d(e_2)=\sum_{i=1}^{n}\beta_i\, e_i.\]
From the equality
\[0=d([e_1,e_2])=[d(e_1),e_2]+[e_1,d(e_2)]=\beta_1\, e_3\]
we get $\beta_1=0$.

Further we have
\[d(e_3)=d([e_1,e_1])=[d(e_1),e_1]+[e_1,d(e_1)]=
(2\alpha_1+\alpha_2)\, e_3+\sum_{i=3}^{n-1}\alpha_i\, e_{i+1}.\] On
the other hand
\[d(e_3)=d([e_2,e_1])=[d(e_2),e_1]+[e_2,d(e_1)]=(\alpha_1+\beta_2)\, e_3
+\sum_{i=3}^{n-1}\beta_i\, e_{i+1}.\]
Therefore, $\beta_2=\alpha_1+\alpha_2, \ \beta_i=\alpha_i, \ \  3\leq i\leq n-1$.

With similar arguments applied on the products $[e_i, e_1]=
e_{i+1}$ and with an induction on $i$, it is easy to check that the
following identities hold for $3 \leq i \leq n$:
\[d(e_i)=\big((i-1)\alpha_1+\alpha_2\big)\, e_i+\sum\limits_{j=i+1}^n\alpha_{j-i+2}\, e_j, \quad 3\leq i\leq n \,.\]
\end{proof}
From Proposition \ref{prop31} we conclude that the  number of
nil-independent outer derivations of the algebra $F_n^1$ is equal
to two. Therefore, by arguments after Proposition \ref{prop2.14}
we have that any solvable Leibniz algebra whose nilradical is
$F_n^1$ has dimension either $n+1$ or $n+2$.

\subsection{Solvable Leibniz algebras with nilradical $F_n^1$}

\

Below we present the  description of such Leibniz algebras when
dimension is equal to $n+1$.

\begin{thm} \label{thm33} An arbitrary $(n+1)$-dimensional solvable Leibniz algebra with nilradical $F_n^1$ is isomorphic to one of the following pairwise non-isomorphic algebras:

\[R_1(\alpha): \left\{\begin{aligned}
{} [e_1,e_1] & =e_3, & [e_i,e_1]& =e_{i+1}, && 2\leq i\leq n-1,\\
[e_1,x]& =-e_1, & [e_2,x]& =-e_2+\alpha\, e_n, && \alpha\in \{0,1\}  \\
[x,e_1] & =e_1,  &  [e_i,x] & =-(i-1)\, e_i, && 3\leq i\leq n \,.
\end{aligned}\right.\]
\[R_2(\alpha_4, \dots, \alpha_{n-1}, \alpha):
\left\{\begin{aligned}
{} [e_1,e_1]& =e_3, & [e_i,e_1]&=e_{i+1}, && 2\leq i\leq n-1,\\
 [e_1,x]&=e_2+\sum\limits_{i=4}^{n-1}\alpha_i\, e_i+\alpha\, e_n, & [e_2,x]&=e_2+\sum\limits_{i=4}^{n-1}\alpha_i\, e_i,&& \\
& & [e_i,x]&=e_i+\sum\limits_{j=i+2}^n\alpha_{j-i+2}\, e_j, && 3\leq i\leq n \,.
\end{aligned}\right.\]
\[R_3(\alpha_4, \dots, \alpha_{n-1}):
\left\{\begin{aligned}
{} [e_1,e_1]&=e_3, & [e_i,e_1]&=e_{i+1}, && 2\leq i\leq n-1,\\
 [e_1,x]& =e_2+\sum\limits_{i=4}^{n-1}\alpha_i\, e_i, & [e_2,x]&=e_2+\sum\limits_{i=4}^{n-1}\alpha_i\, e_i+e_n,&& \\
& & [e_i,x]& =e_i+\sum\limits_{j=i+2}^n\alpha_{j-i+2}\, e_j, && 3\leq i\leq n.
\end{aligned}\right.\]

Moreover,  the first non-vanishing parameter $\{\alpha_4, \dots,
\alpha_{n-1}\}$ in the algebras $R_2(\alpha_4, \dots,
\alpha_{n-1}, \alpha)$ and $R_3(\alpha_4, \dots, \alpha_{n-1})$
can be scaled to 1.
\end{thm}
\begin{proof} From Theorem \ref{thm28} and arguments after Proposition \ref{prop2.14} we know that there exists a basis
$\{e_1, e_2, \dots, e_n, x\}$ such that the multiplication table
of the algebra $F_n^1$ is completed with the products coming from
 $\mathcal{R}_{{x |}_{F_n^1}}(e_i), \  1\leq i\leq n$, i.e.
 \begin{align*}
  [e_1,x]& =\sum\limits_{i=1}^n\alpha_i\, e_i, \qquad [e_2,x]=(\alpha_1+\alpha_2)\, e_2+\sum\limits_{i=3}^{n-1}\alpha_i\, e_i+
\beta\, e_n,\\
[e_i,x] & =\big((i-1)\alpha_1+\alpha_2\big)\, e_i+\sum\limits_{j=i+1}^n\alpha_{j-i+2}\, e_j, \quad 3\leq i\leq n \,.
 \end{align*}
Finally, we consider the remainder products as follows:
\[[x,e_1]=\sum\limits_{i=1}^n\beta_i\, e_i, \qquad  [x,e_2]=\sum\limits_{i=1}^n\gamma_i\, e_i, \qquad
[x,x]=\sum\limits_{i=1}^n\delta_i\, e_i.\]

From the chain of equalities
\[0=[x,e_3]=\big[x,[e_2,e_1]\big]=\big[[x,e_2],e_1\big]-\big[[x,e_1],e_2\big]=\big[[x,e_2],e_1\big]=
(\gamma_1+\gamma_2)\, e_3+\sum\limits_{i=4}^n\gamma_{i-1}\, e_i\] we
conclude that $\gamma_2=-\gamma_1, \   \gamma_i=0, \ 3\leq i\leq
n-1$.

Since $\gamma_1 e_3=\big[e_1,[x,e_2]\big]=\big[[e_1,x],e_2\big]-\big[[e_1,e_2],x\big]=0$, then $\gamma_1=0$.

The identity
\[ \big[e_1,[x,e_1]\big]=\big[[e_1,x],e_1\big]-\big[[e_1,e_1],x\big] \] implies
$\beta_1=-\alpha_1$.

Applying the Leibniz identity to the elements of the form $\{x, x,
e_2\}$ and $\{x, e_2, x\}$, we conclude that:
\[\left\{\begin{aligned}
 \big((n-1)\alpha_1+\alpha_2\big)\gamma_n & =0, \\
 (n-2)\alpha_1\gamma_n & =0\,.
\end{aligned}\right.\]

Note that $\gamma_n=0$ \big(otherwise $\alpha_1=\alpha_2=0$ and then
we get a contradiction with the non-nilpotency of the derivation $D$
(see Proposition \ref{prop31})\big).

Now we are going to discuss the possible cases of the parameters
$\alpha_1$ and $\alpha_2$.

\textbf{Case 1.} Let $\alpha_1\neq0$. Then taking the following
change of basis:
\[x'=-\frac{1}{\alpha_1}x, \ e_1'=e_1-\frac{1}{\alpha_1}\sum\limits_{i=2}^n\beta_i\, e_i, \
e_i'=-\frac{1}{\alpha_1}((-\alpha_1+\beta_2) \, e_i
+\sum\limits_{j=i+1}^n\beta_{j-i+2}\, e_j), \ 2\leq i\leq n,\]
we obtain
\begin{align*}
[e_1,e_1]& =e_3, & [e_1,x] & =\sum\limits_{i=1}^n\mu_i\, e_i, & [e_i,e_1]& =e_{i+1},&& 2\leq i\leq n-1, \\
 [x,e_1] & =e_1, &
[e_2,x] & =\sum\limits_{i=1}^n\eta_i\, e_i, & [x,e_2] & =0, &
[x,x] =\sum\limits_{i=1}^n\theta_i\, e_i \,.
\end{align*}

 From the equalities
\[0=\big[[e_1,e_2],x\big]=\big[e_1,[e_2,x]\big]+\big[[e_1,x],e_2\big]=\big[e_1,\sum\limits_{i=1}^n\eta_i\,e_i\big]=\eta_1\, e_3 \ \text{we have} \ \eta_1=0\,.\]

Consider
\begin{align*}
[e_3,x]& {}= {} \big[[e_1,e_1],x\big]=\big[e_1,[e_1,x]\big]+\big[[e_1,x],e_1\big]\\
& {}= {} \mu_1\, e_3+(\mu_1+\mu_2)\, e_3+\sum_{i=3}^{n-1}\mu_i\, e_{i+1}=(2\mu_1+\mu_2)\, e_3+\sum_{i=3}^{n-1}\mu_i\, e_{i+1}\,.
\end{align*}
On the other hand
\[ [e_3,x]=\big[[e_2,e_1],x\big]=\big[e_2,[e_1,x]\big]+\big[[e_2,x],e_1\big]=
\mu_1\, e_3+\eta_2\, e_3+\sum_{i=3}^{n-1}\eta_i\, e_{i+1}=
(\mu_1+\eta_2)\, e_3+\sum_{i=3}^{n-1}\eta_i\, e_{i+1}.\] The
comparison of both linear combinations implies that:
\[\eta_2=\mu_1+\mu_2,\qquad \eta_i=\mu_i, \quad 3\leq i\leq n-1,\]
 that it  is to say:
\[[e_2,x]=(\mu_1+\mu_2)\, e_2+\sum\limits_{i=3}^{n-1}\mu_i\, e_i+\eta_n\, e_n \qquad
\text{and} \qquad [e_3,x]= (2\mu_1+\mu_2)\, e_3+\sum_{i=3}^{n-1}\mu_i\,
e_{i+1}.\] Now we shall prove the following equalities by an induction on
$i$:
\begin{equation} \label{E:ind}
[e_i,x]=\big((i-1)\mu_1+\mu_2\big)\, e_i+\sum\limits_{j=i+1}^n\mu_{j-i+2}\, e_j, \quad 3\leq i\leq n.
\end{equation}
Obviously, the equality holds for $i=3$. Let us assume that the
equality holds for $3< i < n$, and we prove it for $i+1$:
\begin{align*}
[e_{i+1},x]& {}={}\big[[e_i,e_1],x\big]=\big[e_i,[e_1,x]\big]+\big[[e_i,x],e_1\big] \\
& {}={} \mu_1\, e_{i+1}+\big((i-1)\mu_1+\mu_2\big)\, e_{i+1}+
\sum\limits_{j=i+2}^n\mu_{j-i+1}\, e_j=
(i\mu_1+\mu_2)\, e_{i+1}+\sum\limits_{j=i+2}^n\mu_{j-i+1}\, e_j \,;
\end{align*}
 so the induction proves the equalities \eqref{E:ind} for any $i, \ 3\leq i\leq n$.

 Applying the Leibniz identity to the elements $\{e_1, x,
e_1\}, \ \{e_1, x, x\}, \ \{x, e_1, x\}$, we deduce that:
\[\mu_1=-1, \qquad \mu_2=\theta_1=0, \qquad \theta_i=\mu_{i+1}, \quad
2\leq i\leq n-1.\]

Below we summarize the table of multiplication of the algebra
\[\left\{ \begin{aligned}
{} [e_1,e_1]&=e_3 & [e_i,e_1]& =e_{i+1}, && 2\leq i\leq n-1, \\
[e_1,x]& =-e_1+\sum\limits_{i=3}^n\mu_i\, e_i, & [e_2,x]& =-e_2+\sum\limits_{i=3}^{n-1}\mu_i\, e_i+\eta_n\, e_n, &&\\
[x,e_1]& =e_1, & [e_i,x]& =-(i-1)\,e_i+\sum\limits_{j=i+1}^{n}\mu_{j-i+2}\, e_j, &&  3\leq i\leq n, \\
  [x,x]& =\sum\limits_{i=2}^{n-1}\mu_{i+1}\, e_i+\theta_n\, e_n. &&  & &
\end{aligned} \right.\]

Let us take the change of basis in the following form:
\[e_1'=e_1+\sum\limits_{i=3}^nA_ie_i, \ e_2'=e_2+\sum\limits_{i=3}^nA_ie_i, \ e_i'=e_i+\sum\limits_{j=i+1}^nA_{j-i+2}e_j,\
 3\leq i\leq n, \ x'=\sum\limits_{i=2}^{n-1}A_{i+1}e_i+Be_n+x.\]
where \[A_3=\mu_3,\  A_i=\frac{1}{(i-2)}\Big(\mu_i+\sum\limits_{j=3}^{i-1}A_j\mu_{i-j+2}\Big),\ 4\leq i\leq n \ \ \text{and}
 \ \ B=\frac{1}{n-1}\Big(\theta_n+\sum\limits_{j=3}^nA_j\mu_{n-j+3}\Big).\]

Then

\begin{align*}
[x',e_1'] & = \Big[\sum\limits_{i=2}^{n-1}A_{i+1}e_i+Be_n+x,e_1\Big]=e_1+\sum\limits_{i=3}^nA_ie_i=e_1', \\
[e_1',x'] & =  [e_1,x]+\sum\limits_{i=3}^nA_i[e_i,x]=
-e_1+\sum\limits_{i=3}^n\mu_ie_i+
\sum\limits_{i=3}^nA_i\,\Big(-(i-1)e_i+\sum\limits_{j=i+1}^{n}\mu_{j-i+2}e_j\Big) \\
&{}=-e_1-\sum\limits_{i=3}^nA_ie_i+\sum\limits_{i=3}^n\mu_ie_i-\sum\limits_{i=3}^nA_i(i-2)e_i+
\sum\limits_{i=3}^nA_i \,\Big(\sum\limits_{j=i+1}^{n}\mu_{j-i+2}e_j\Big) \\
&{}=-e_1-\sum\limits_{i=3}^nA_ie_i+
\sum\limits_{i=3}^n\mu_ie_i-\sum\limits_{i=3}^nA_i(i-2)e_i+
\sum\limits_{i=4}^n\Big(\sum\limits_{j=3}^{i-1}A_j\mu_{i-j+2}\Big)e_i\\
&{}= -e_1-\sum\limits_{i=3}^nA_ie_i+(\mu_3-A_3)e_3+
\sum\limits_{i=4}^n\Big(-A_i(i-2)+\mu_i+
\sum\limits_{j=3}^{i-1}A_j\mu_{i-j+2}\Big)e_i\\
&{}= -e_1-\sum\limits_{i=3}^nA_ie_i=-e_1' \\
[e_2',x']=& [e_2,x]+\sum\limits_{i=3}^nA_i[e_i,x]=-e_2+\sum\limits_{i=3}^{n-1}\mu_ie_i+\eta_n
e_n+\sum\limits_{i=3}^nA_i \, \Big(-(i-1)e_i+\sum\limits_{j=i+1}^{n}\mu_{j-i+2}\,e_j\Big)\\
&{}=-e_2-\sum\limits_{i=3}^nA_ie_i+\sum\limits_{i=3}^{n-1}\mu_ie_i+\eta_n
e_n- \sum\limits_{i=3}^nA_i(i-2)e_i+\sum\limits_{i=3}^nA_i\,\Big(\sum\limits_{j=i+1}^{n}\mu_{j-i+2}\,e_j\big)\\
&{}=-e_2-\sum\limits_{i=3}^nA_ie_i+\sum\limits_{i=3}^{n-1}\mu_ie_i+\eta_ne_n-
\sum\limits_{i=3}^nA_i(i-2)e_i+\sum\limits_{i=4}^n \,\Big(\sum\limits_{j=3}^{i-1}A_j\mu_{i-j+2}\Big)\,e_i\\
&{}=-e_2-\sum\limits_{i=3}^nA_ie_i+(\mu_3-A_3)e_3+\sum\limits_{i=4}^{n-1}\, \Big(-A_i(i-2)+
\mu_i+\sum\limits_{j=3}^{i-1}A_j\mu_{i-j+2}\Big)\,e_i \\
&{} \quad +\,\Big(\eta_n-(n-2)A_n + \sum\limits_{i=3}^{n-1}A_i\mu_{n-i+2}\Big)\,e_n= -e_2'+\eta'\,e_n',
\end{align*}
\begin{align*}
 [x',x']& = \sum\limits_{i=2}^{n-1}A_{i+1}[e_i,x]+B[e_n,x]+[x,x]\\
&{}= \sum\limits_{i=2}^{n-1}A_{i+1}\,\Big(-(i-1)e_i+\sum\limits_{j=i+1}^{n}\mu_{j-i+2}e_j\Big)-
B(n-1)e_n+\sum\limits_{i=2}^{n-1}\mu_{i+1}e_i + \theta_ne_n\\
&{}= -\sum\limits_{i=2}^{n-1}A_{i+1}(i-1)e_i+\sum\limits_{i=2}^{n-1}\mu_{i+1}e_i- B(n-1)e_n+\theta_ne_n+
\sum\limits_{i=2}^{n-1}A_{i+1}\,\Big(\sum\limits_{j=i+1}^{n}\mu_{j-i+2}e_j\Big)\\
&{} = -\sum\limits_{i=2}^{n-1}A_{i+1}(i-1)e_i+\sum\limits_{i=2}^{n-1}\mu_{i+1}e_i-B(n-1)e_n+\theta_n e_n+
\sum\limits_{i=3}^n \, \Big(\sum\limits_{j=3}^iA_j\mu_{i-j+3}\Big)\, e_i \\
&{}= (\mu_3-A_3)e_2+ \sum\limits_{i=3}^{n-1}\Big(-A_{i+1}(i-1)+\mu_{i+1}+
\sum\limits_{j=3}^iA_j\mu_{i-j+3}\Big)\,e_i\\
&{} \quad +\Big(-B(n-1)+\theta_n+\sum\limits_{j=3}^nA_j\mu_{n-j+3}\Big)\,e_n=0\,.
\end{align*}

With a similar induction as the given for  equations \eqref{E:ind},  it is
easy to check that the following equalities hold:
\[[e_i,x]=-(i-1)\, e_i, \qquad  3\leq i\leq n.\]

Thus, we obtain the following table of multiplication:
\[\left\{\begin{aligned}
{}[e_1,e_1]&=e_3, & [e_i,e_1]&=e_{i+1}, && 2\leq i\leq n-1,\\
[e_1,x]& =-e_1, & [e_2,x]& =-e_2+\eta e_n, &&\\
 [x,e_1]& =e_1,  & [e_i,x]& =-(i-1)e_i, && 3\leq i\leq n.
\end{aligned}\right.\]

Now we take the general change of basis in the following form:

\begin{equation} \label{E:gchbasis}
\left\{\begin{aligned}
 e_1'& = \sum\limits_{i=1}^nA_ie_i, && \\
 e_i'& =A_1^{i-2}\Big((A_1+A_2) \, e_i+\sum\limits_{j=i+1}^nA_{j-i+2}\, e_j\Big), &&  2\leq i\leq n,\\
 x'&=\sum\limits_{i=1}^nB_ie_i+B_{n+1}x, && \text{where} \  A_1(A_1+A_2)B_{n+1}\neq 0 \,.
\end{aligned}\right.
\end{equation}

Then from the equalities
\[\sum\limits_{i=1}^nA_ie_i=e_1'=[x',e_1']=A_1B_{n+1}e_1+A_1(B_1+B_2)e_3+
A_1 \, \big(\sum\limits_{j=4}^nB_{j-1}e_j\big)\]
we obtain:
\[B_{n+1}=1, \qquad A_2=0, \qquad  B_1+B_2=A_3/A_1,\qquad  B_i=A_{i+1}/A_1, \quad 3\leq
i\leq n-1.\]

Similarly, from
\[-A_1e_1-\sum\limits_{i=3}^nA_ie_i=-e_1'=[e_1',x']=
-A_1e_1+(A_1B_1-2A_3)e_3+\sum\limits_{i=4}^n\big(B_1A_{i-1}-(i-1)A_i\big)e_i\]
we obtain: \[B_1=A_3/A_1, \qquad A_i=\frac{A_3A_{i-1}}{(i-2)A_1}, \quad
4\leq i\leq n,\] consequently \[B_2=0, \qquad  \qquad
A_i=\frac{A_3^{i-2}}{(i-2)!A_1^{i-3}}, \quad 4\leq i\leq n.\]

Consider the product $[e_2',x']$, namely:

\begin{align*}
[e_2',x']& =\Big[A_1e_2+A_3e_3+\sum\limits_{i=4}^n\frac{A_3^{i-2}}{(i-2)!A_1^{i-3}}e_i,
\frac{A_3}{A_1}e_1+x\Big]\\
&{}=A_3e_3+\frac{A_3^2}{A_1}e_4+
\sum\limits_{i=5}^n\frac{A_3^{i-2}}{(i-3)!A_1^{i-3}}e_i-A_1e_2+A_1\eta e_n-2A_3e_3-
\sum\limits_{i=4}^n\frac{A_3^{i-2}(i-1)}{(i-2)!A_1^{i-3}}e_i\\
&{}=-A_1e_2-A_3e_3-\frac{A_3^2}{2A_1}e_4-\sum\limits_{i=5}^{n-1}\frac{A_3^{i-2}}{(i-2)!A_1^{i-3}}e_i+
\Big(A_1\eta-\frac{A_3^{n-2}}{(n-2)!A_1^{n-3}}\Big)e_n \,.
\end{align*}

On the other hand, we have
\[[e_2',x']=-e_2'+\eta'e_n'=-A_1e_2-A_3e_3-\frac{A_3^2}{2A_1}e_4-
\sum\limits_{i=4}^{n}\frac{A_3^{i-2}}{(i-2)!A_1^{i-3}}e_i+\eta'A_1^{n-1}e_n.\]

Therefore, the parameter $\eta'$ satisfies the relation
$\eta'=\dfrac{1}{A_1^{n-2}}\eta$.

If $\eta=0$, then $\eta'=0$. If $\eta\neq0$, then choosing $A_1$ such that $A_1^{n-2}=\eta$, we conclude
 $\eta'=1$. Thus the algebras $R_1(\alpha)$, for $\alpha\in\{0, 1\}$, are obtained.

\textbf{Case 2.} Let $\alpha_1=0,\alpha_2\neq 0$. Then making
the following change of  basis
\[x'=x-\sum\limits_{i=2}^{n-1}\beta_{i+1}\, e_i\] we can assume that
$[x,e_1]=\beta_2\, e_2$.

From the identity \[ \big[x,[x,e_1]\big]=\big[[x,x],e_1\big]-\big[[x,e_1],x\big] \] we derive
\[0=\sum\limits_{i=3}^n\delta_{i-1}\, e_i-\beta_2[e_2,x]=
\sum\limits_{i=3}^n\delta_{i-1}\, e_i-
\beta_2 \,\Big(\sum\limits_{i=2}^{n-1}\alpha_i\, e_i+\beta \, e_n\Big),\]
consequently, $\beta_2=0, \ \delta_i=0, \ 2\leq i\leq n-1$.

Making the change of basis
\[x'=x-\frac{\delta}{\alpha_2}\, e_n\]
we can assume that $[x, x]=0$.

Summarizing, we obtain the following table of multiplication of
the algebra in this case:
\[\left\{\begin{aligned}
{} [e_1,e_1] & =e_3, & [e_i,e_1] & =e_{i+1}, && 2\leq i\leq n-1,\\
 [e_1,x]& =\sum\limits_{i=2}^n\alpha_i\, e_i,  & [e_2,x]& =\sum\limits_{i=2}^{n-1}\alpha_i\, e_i+\beta\, e_n,&& \\
 [e_i,x]&=\sum\limits_{j=i}^n\alpha_{j-i+2}\, e_j,  & &  && 3\leq i\leq n\,.
\end{aligned}\right.\]

Now we shall study  the behaviour of the parameters in this family
of algebras under the general change of basis in the form \eqref{E:gchbasis}.

Then the equalities
\[0=[x',e_1']=\Big[\sum\limits_{i=1}^nB_ie_i+B_{n+1}x,A_1e_1\Big]=
A_1 \,\Big((B_1+B_2)e_3+\sum\limits_{i=4}^nB_{i-1}e_i\Big)\] imply $B_1=-B_2,
\ B_i=0, \ 3\leq i\leq n-1$.

Now we shall express the product $[e_1',x']$ as a linear
combination of the basis $\{e_1, e_2, \dots, e_n, x\}$, namely:
\begin{align*}
[e_1',x'] & =  \Big[\sum\limits_{i=1}^nA_ie_i,B_1e_1+B_{n+1}x\Big]\\
& {} =  B_1 \Big ((A_1+A_2)e_3+\sum\limits_{i=4}^nA_{i-1}e_i\Big) \\
& {} \quad  + B_{n+1} \, \bigg(A_1\sum\limits_{i=2}^n\alpha_ie_i+A_2 \,\Big(\sum\limits_{i=2}^{n-1}\alpha_ie_i+\beta
e_n\Big)+\sum\limits_{i=3}^nA_i\sum\limits_{j=i}^n\alpha_{j-i+2}e_j\bigg)\\
& {} =  B_1(A_1+A_2)e_3+\sum\limits_{i=4}^nB_1A_{i-1}e_i \\
& {} \quad + B_{n+1}A_1\sum\limits_{i=2}^n\alpha_ie_i+B_{n+1}A_2\sum\limits_{i=2}^{n-1}\alpha_ie_i+B_{n+1}A_2\beta
e_n+ B_{n+1}\sum\limits_{i=3}^n\sum\limits_{j=3}^iA_j\alpha_{i-j+2}e_i\\
& {} = B_{n+1}(A_1+A_2)\alpha_2\,e_2+\Big((A_1+A_2)(B_1+B_{n+1}\alpha_3)
 + B_{n+1}A_3\alpha_2\Big)\,e_3 \\
 & {} \quad + \sum\limits_{i=4}^{n-1}\, \Big(B_1A_{i-1}+B_{n+1}(A_1+A_2)\alpha_i
 + \sum\limits_{j=3}^iB_{n+1}A_j\alpha_{i-j+2}\Big)\,e_i \\
& {} \quad  + \bigg(B_1A_{n-1}+ B_{n+1}\Big(A_1\alpha_n+A_2\beta+\sum\limits_{i=3}^nA_i\alpha_{n-i+2}\Big)\bigg)\,e_n \,.
\end{align*}

On the other hand
\begin{align*}
[e_1',x'] & = \sum\limits_{i=2}^n\alpha_i'e_i'=
\sum\limits_{i=2}^n\alpha_i'A_1^{i-2} \, \Big((A_1+A_2)e_i+\sum\limits_{j=i+1}^nA_{j-i+2}e_j\Big) \\
& {} =  \sum\limits_{i=2}^n\alpha_i'A_1^{i-2}(A_1+A_2)e_i+
\sum\limits_{i=3}^n\sum\limits_{j=3}^iA_1^{i-j}A_j\alpha_{i-j+2}'e_i \\
& {} =  \alpha_2'(A_1+A_2)e_2+\big(A_1(A_1+A_2)\alpha_3'+A_3\alpha_2'\big)e_3+
\sum\limits_{i=4}^n\, \Big(\alpha_i'A_1^{i-2}(A_1+A_2)+
\sum\limits_{j=3}^iA_1^{i-j}A_j\alpha_{i-j+2}'\Big)e_i \,.
\end{align*}

Comparing coefficients at the basic elements in both combinations,
we obtain the following relations:
\begin{align*}
\alpha_2'(A_1+A_2)& =  B_{n+1}(A_1+A_2)\alpha_2,\\
A_1(A_1+A_2)\alpha_3'+A_3\alpha_2'& =(A_1+A_2)(B_1+B_{n+1}\alpha_3)+B_{n+1}A_3\alpha_2,\\
\alpha_i'A_1^{i-2}(A_1+A_2)+\sum\limits_{j=3}^iA_1^{i-j}A_j\alpha_{i-j+2}'&=
B_1A_{i-1}+B_{n+1}(A_1+A_2)\alpha_i+\sum\limits_{j=3}^iB_{n+1}A_j\alpha_{i-j+2}, \\
&  \qquad   \qquad \qquad   \qquad   \qquad   \qquad  \qquad   \qquad 4\leq i\leq n-1,\\
A_1^{n-2}\alpha_n'(A_1+A_2)+\sum\limits_{j=3}^nA_1^{n-j}A_j\alpha_{n-j+2}'& =B_1A_{n-1}+
B_{n+1}\Big(A_1\alpha_n+A_2\beta+\sum\limits_{i=3}^nA_i\alpha_{n-i+2}\Big)\,.
\end{align*}

The simplification of these relations implies the following
identities:

\[\alpha_2'=B_{n+1}\alpha_2,\quad \alpha_3'=\frac{B_1+\alpha_3B_{n+1}}{A_1},\quad \alpha_i'=\frac{B_{n+1}\alpha_i}{A_1^{i-2}}, \ 4\leq i\leq n-1,\quad \alpha_n'=\frac{(\alpha_nA_1+\beta A_2)B_{n+1}}{A_1^{n-2}(A_1+A_2)}.\]

Analogously, considering the product $[e_2',x']$, we get the
relation:
\[\beta'=\frac{\beta B_{n+1}}{A_1^{n-2}},\] and \[[x',x']=-\frac{(\beta B_1-\alpha_nB_1-\alpha_2B_n)B_{n+1}}{A_1^{n-2}(A_1+A_2)}e_n.\]

 Since $[x',x']=0$, then $B_n=\dfrac{\beta
B_1-\alpha_nB_1}{\alpha_2 }$.

 Setting $B_{n+1}=1/\alpha_2$ and
$B_1=-\alpha_3/\alpha_2$, then we derive that $\alpha_2'=1, \
\alpha_3'=0$.

 If $\beta=0$ and $\alpha_n=0$, then $\beta'=0$ and we
obtain the algebra $R_2(\alpha_4, \dots, \alpha_{n-1},0)$.

If $\beta=0$ and $\alpha_n\neq0$, then putting
$A_2=\dfrac{\alpha_n-\alpha_2A_1^{n-2}}{\alpha_2A_1^{n-3}}$, we have
 $\alpha_n'=1$ and so we obtain the algebra $R_2(\alpha_4, \dots, \alpha_{n-1},1)$.

 If $\beta\neq0$, then choosing
\[A_1=\sqrt[\uproot{4} n-2]{\dfrac{\beta}{\alpha_2}},\qquad A_2=-\dfrac{A_1\alpha_n}{\beta} \,, \]
we obtain $\beta'=1, \ \alpha_n'=0$ and the algebra $R_3(\alpha_4,
\dots, \alpha_{n-1})$.
\end{proof}

Now we shall consider the case when the dimension of a solvable
Leibniz algebra with nilradical $F_n^1$ is equal to $n+2$.

\begin{thm} It does not exist any $(n+2)$-dimensional solvable Leibniz algebra with nilradical $F_n^1$.
\end{thm}
\begin{proof} From the conditions of the theorem, we have the  existence of a basis $\{e_1, e_2, \dots, e_n, x, y\}$
such that the table of multiplication of $F_n^1$ remains. The
outer non-nilpotent derivations of $F_n^1$, denoted by $\mathcal{R}_{x_{\mid
F_n^1}}$ and $\mathcal{R}_{y_{\mid F_n^1}}$, are  of the form given in
Proposition \ref{prop31}, with the set of entries $\{\alpha_i,
\gamma\}$ and $\{\beta_i, \delta\}$, respectively, where
$[e_i,x]=\mathcal{R}_{x_{\mid F_n^1}}(e_i)$ and $[e_i,y]=\mathcal{R}_{y_{\mid
F_n^1}}(e_i)$.

Taking the following change of basis:
\begin{equation}\label{E:chbasis}
 x'=\frac{\beta_2}{\alpha_1\beta_2-\alpha_2\beta_1}x-
\frac{\alpha_2}{\alpha_1\beta_2-\alpha_2\beta_1}y,\quad \quad y'=-\frac{\beta_1}{\alpha_1\beta_2-\alpha_2\beta_1}x+
\frac{\alpha_1}{\alpha_1\beta_2-\alpha_2\beta_1}y,
\end{equation}
we may assume that $\alpha_1=\beta_2=1$ and $\alpha_2=\beta_1=0$.

Therefore we have the products
\begin{align*}
 [e_1,x]&=e_1+\sum\limits_{i=3}^n\alpha_ie_i, &
[e_2,x]&=e_2+\sum\limits_{i=3}^{n-1}\alpha_ie_i+\gamma e_n, &
[e_i,x]&=(i-1)e_i+\sum\limits_{j=i+1}^n\alpha_{j-i+2}e_j, && 3\leq
i\leq n,\\
[e_1,y]&=e_2+\sum\limits_{i=3}^n\beta_ie_i, &
[e_2,y]&=e_2+\sum\limits_{i=3}^{n-1}\beta_ie_i+\delta e_n, &
[e_i,y]&=e_i+\sum\limits_{j=i+1}^n\beta_{j-i+2}e_j, && 3\leq i\leq
n, \end{align*}

Applying similar arguments as in Case 1 of Theorem
\ref{thm33} and taking into account that the products $[e_1, y], \
[e_2, y], \ [e_i, y]$ will not be changed under the bases
transformations which were used there, we obtain the products:
\[[e_1,x]=e_1, \quad [e_2,x]=e_2+\gamma e_n, \quad  [e_i,x]=(i-1)e_i, \ \ 3\leq i\leq n, \qquad  [x,e_1]=-e_1 \,.\]

Let us introduce the notations:
\[[y,e_1]=\sum\limits_{i=1}^n\eta_ie_i, \quad [y,e_2]=\sum\limits_{i=1}^n\theta_ie_i, \quad [y,y]=\sum\limits_{i=1}^n\tau_ie_i, \quad [x,y]=\sum\limits_{i=1}^n\sigma_ie_i, \quad [y,x]=\sum\limits_{i=1}^n\rho_ie_i.\]

From the Leibniz identity
\[ \big[e_1,[y,e_1]\big]=\big[[e_1,y],e_1\big]-\big[[e_1,e_1],y\big] \] we get $\eta_1=0$.

Note that we can assume $[y,e_1]=\eta_2e_2$ (by changing $y'=y-\sum\limits_{i=2}^{n-1}\eta_{i+1}e_i$).

Due to \[ \big[y,[e_1,e_2]\big]=\big[[y,e_1],e_2\big]-\big[[y,e_2],e_1\big] \] we obtain
$\theta_2=-\theta_1, \ \theta_i=0,\  3\leq i\leq n-1$.

Since $\big[e_1,[y,e_2]\big]=\big[[e_1,y],e_2\big]-\big[[e_1,e_2],y\big]$, then we have
$\theta_1=\theta_2=0$. Moreover,  the Leibniz identity
$\big[y_1,[y,e_2]\big]=\big[[y,y],e_2\big]-\big[[y,e_2],y\big]$ implies that $\theta_n=0$,
i.e., $[y,e_2]=0$.

From the following chain of equalities
\begin{align*}
0&{}={}\eta_2[y,e_2]=[y,\eta_2e_2]=\big[y,[y,e_1]\big]=\big[[y,y],e_1\big]-\big[[y,e_1],y\big]\\
&{}={}(\tau_1+\tau_2)e_3+
\sum\limits_{i=4}^n\tau_{i-1}e_i-\eta_2[e_2,y]=(\tau_1+\tau_2)e_3+
\sum\limits_{i=4}^n\tau_{i-1}e_i-\eta_2\Big(e_2+\sum\limits_{i=3}^{n-1}\beta_ie_i+\delta
e_n\Big)
\end{align*}
 we derive that \[\eta_2=0, \qquad \tau_2=-\tau_1, \qquad  \tau_i=0, \quad
3\leq i\leq n-1.\]

Therefore, we have $[y,e_1]=0$ and $[y,y]=\tau_1e_1-\tau_1e_2+ \tau_ne_n$.

Consider the Leibniz identity

\[ \big[x,[y,e_1]\big]=\big[[x,y],e_1\big]-\big[[x,e_1],y\big] \] then we get
\[-e_2-\sum\limits_{i=3}^n\beta_ie_i=
(\sigma_1+\sigma_2)e_3+\sum\limits_{i=3}^{n-1}\sigma_{i}e_{i+1}.\]

Thus, we have a contradiction with the assumption of the existence
of an algebra under the conditions  of the  theorem.
\end{proof}

\subsection{Solvable Leibniz algebras with nilradical $F_n^2$}

\

In this section we describe solvable Leibniz algebras with
nilradical $F_n^2$, i.e. solvable Leibniz algebras $R$ which
decompose in the form $R=F_n^2 \oplus Q$.

\begin{prop} \label{prop5.5} An arbitrary derivation of the algebra $F_n^2$ has the following matrix form:
 \[
D=\begin{pmatrix}
\alpha_1& \alpha_2&\alpha_3&\alpha_4&\dots&\alpha_{n-1}&\alpha_n\\
0& \beta &0&0&\dots&0& \gamma\\
0& 0&2\alpha_1&\alpha_3&\dots&\alpha_{n-2}& \alpha_{n-1}\\
0& 0&0&3\alpha_1&\dots&\alpha_{n-3}&\alpha_{n-2}\\
\vdots&\vdots&\vdots&\vdots&\dots&\vdots&\vdots\\
0&0&0&0&\dots&0&(n-1)\alpha_1
\end{pmatrix}.
\]
\end{prop}
\begin{proof} The proof follows by straightforward calculations in a similar way as the proof of Proposition \ref{prop31}.
\end{proof}

\begin{rem}
 It is an easy task  to check  that the number of nil-independent
derivations of the algebra  $F_n^2$ is equal to 2.
\end{rem}

\begin{cor} The dimension of a solvable Leibniz algebra with nilradical $F_n^2$ is either $n+1$ or $n+2$.
\end{cor}
\begin{thm} \label{thm5.7} An $(n+1)$-dimensional  solvable Leibniz algebra with nilradical $F_n^2$
is isomorphic to one of the following pairwise non-isomorphic algebras:
\[R_1(\alpha):  \left\{\begin{aligned}
{} [e_1,e_1] & =e_3, & [e_i,e_1]& =e_{i+1}, && 3\leq i\leq n-1,\\
[e_1,x]& =-e_1,  & [e_i,x]& =-(i-1)\, e_i, && 3\leq i\leq n, \\
 [x,e_1]& =e_1, & [x,x] & =\alpha\, e_2, && \alpha\in\{0,1\}\,.
\end{aligned}\right.\]
\[ R_2(\alpha): \left\{\begin{aligned}
{} [e_1,e_1]& =e_3, & [e_i,e_1]& =e_{i+1}, && 3\leq i\leq n-1,\\
[e_1,x]&=-e_1,& [e_i,x]&=-(i-1)\, e_i, && 3\leq i\leq n, \\
 [x,e_1]& =e_1,  &[e_2,x]&=\alpha\, e_2, && \alpha\neq 0 \,.
\end{aligned}\right.\]
\[ R_3: \left\{\begin{aligned}
{} [e_1,e_1]& =e_3, & [e_i,e_1]& =e_{i+1}, && 3\leq i\leq n-1,\\
[e_1,x]&=-e_1,& [e_i,x]&=-(i-1)\, e_i, && 3\leq i\leq n, \\
 [x,e_1]& =e_1,  &[e_2,x]&=(1-n)\, e_2+e_n\,.
\end{aligned}\right.\]
\[R_4(\alpha): \left\{\begin{aligned}
{} [e_1,e_1]& =e_3, & [e_i,e_1]&=e_{i+1}, && 3\leq i\leq n-1,\\
[e_1,x]&=-e_1,& [e_i,x]&=-(i-1)\, e_i, && 3\leq i\leq n,\\
 [x,e_1]&=e_1, & [e_2,x]& =-\alpha\, e_2, && \alpha\neq 1,\\
 [x,e_2]& =\alpha\, e_2 \,. &  & &&
\end{aligned}\right.\]
\[R_5(\alpha):\left\{\begin{aligned}
{} [e_1,e_1]& =e_3, && & [e_i,e_1]& =e_{i+1}, && 3\leq i\leq n-1,\\
[e_1,x]& =-e_1-\alpha\, e_2,&& \alpha\in\{0,1\},& [e_i,x]& =-(i-1)\, e_i, && 3\leq i\leq n,\\
[x,e_1]&=e_1+\alpha\, e_2, &&    & [e_2,x]& =-e_2, & &\\
 [x,e_2]& =e_2\,. & & & &  & &
\end{aligned}\right.\]
\[R_6(\alpha_3,\alpha_4, \dots,\alpha_n,\lambda, \delta):
\left\{\begin{aligned}
{} [e_1,e_1]& =e_3, & [e_i,e_1]& =e_{i+1}, && 3\leq i\leq n-1,\\
 [e_1,x]&=\sum\limits_{i=3}^n\alpha_i\, e_i, & [e_i,x]&=\sum\limits_{j=i+1}^n\alpha_{j-i+2}\, e_j, && 3\leq i\leq n-1,\\
  [x,x]& =\lambda\, e_n, & [e_2,x]&=e_2, &&  \\
 [x,e_2]&=\delta\, e_2, & & & &  \delta\in\{0, -1\} \,.
\end{aligned}\right.\]
In the algebra $R_6(\alpha_3, \alpha_4, \dots,\alpha_n,\lambda,
\delta)$ the first non vanishing parameter
$\{\alpha_3,\alpha_4,\dots,\alpha_n,\lambda\}$ can be scaled to 1.
\end{thm}
\begin{proof} Let $R$ be a solvable Leibniz algebra satisfying the conditions of the theorem, then
there exists a basis $\{e_1, e_2, \dots, e_n, x\}$, such that
$\{e_1, e_2, \dots, e_n\}$ is the standard basis of $F_n^2$, and
for non nilpotent outer derivations of the algebra $F_n^2$ we have
that $[e_i,x]=\mathcal{R}_{x_{\mid F_n^2}}(e_i), \ 1\leq i\leq n$,

Due to Proposition \ref{prop5.5} we can assume that
\[ [e_1,x]=\sum\limits_{i=1}^n\alpha_i\,e_i, \quad [e_2,x]=\beta_2\, e_2+\beta_n\, e_n, \quad
[e_i,x]=(i-1)\alpha_1\, e_i+\sum\limits_{j=i+1}^n\alpha_{j-i+2}\,e_j,\quad 3\leq i\leq n.\]

Let us introduce the following notations:
\[[x,e_1]=\sum\limits_{i=1}^n\gamma_i\, e_i, \qquad [x,e_2]=\sum\limits_{i=1}^n\delta_i\, e_i, \qquad
[x,x]=\sum\limits_{i=1}^n\lambda_i\, e_i.\]

Considering the Leibniz identity for the elements $\{e_1, x, x\},
\ \{e_1, x, e_1\}, \ \{x, e_2, e_1\}$ we obtain $\lambda_1=0, \
\gamma_1=-\alpha_1$ and $[x,e_2]=\delta_2e_2+\delta_ne_n$. By
setting $e_2'=\delta_2e_2+\delta_ne_n$ we can assume that
$[x,e_2]=\delta e_2$.

Now we distinguish the following possible cases:

\textbf{Case 1.} Let $\alpha_1\neq 0$. Then the following change of basis
\[x'=\frac{1}{\gamma_1}\, x, \quad e_1'=e_1+\frac{1}{\gamma_1}\sum\limits_{j=3}^n\gamma_{j}e_j, \quad e_2'=e_2,
\quad
e_i'=e_i+\frac{1}{\gamma_1}\sum\limits_{j=i+1}^n\gamma_{j-i+2}\,e_j, \  \ 3\leq i\leq n,\]
 implies that $[x',e_1']=e_1'+\gamma e_2'$
(where $\gamma=\frac{\gamma_2}{\gamma_1}$) and the rest of
products remains unchanging.

From the equalities:
\[e_1+\gamma(1+\delta)\, e_2=\big[x,[x,e_1]\big]=\big[[x,x],e_1\big]-\big[[x,e_1],x\big]=
\sum\limits_{i=4}^n\lambda_{i-1}\, e_i-\sum\limits_{i=1}^n\alpha_i\,e_i- \gamma\beta_2\, e_2-\gamma\beta_n\, e_n\]
we deduce that
\[\alpha_1=-1, \ \alpha_3=0, \ \alpha_2=-\gamma(1+\delta+\beta_2), \
\lambda_i=\alpha_{i+1}, \ \ 3\leq i\leq n-2,  \quad \text{and} \quad
\lambda_{n-1}=\alpha_n+\gamma\beta_n \,.\]
In addition, if we take the following change of basis:
\[e_1'=e_1+\sum\limits_{i=4}^nA_i\, e_i, \quad e_2'=e_2, \quad e_i'=e_i+\sum\limits_{j=i+2}^nA_{j-i+2}\, e_j, \ 3\leq i\leq n,
 \quad x'=\sum\limits_{i=3}^{n-1}A_{i+1}\, e_i+B\, e_n+x,\]
where $A_j=\dfrac{1}{2}\alpha_j, \ j=4,5, \quad  A_i=\dfrac{1}{i-2}\big(\alpha_i+\sum\limits_{j=4}^{i-2}A_j\alpha_{i-j+2}\big), \ \ 6\leq i \leq n$ and
$B=\dfrac{1}{n-1}\big(\lambda_n+\sum\limits_{j=4}^{n-1}A_j\alpha_{i-j+3}\big)$,
then we have
\[[e_1',x']=-e_1'+\alpha_2\, e_2', \ [e_2',x']=\beta_2\, e_2'+\beta_n\, e_n', \
[x',x']=\lambda_2\, e_2'+\gamma\beta_n\, e_{n-1}', \
[e_i',x']=-(i-1)\, e_i', \ 3\leq i\leq n.\]

Finally, we obtain the following table of multiplication of the
algebra $R$:
\[\left\{\begin{aligned}
{} [e_1,x]& =-e_1-\gamma(1+\delta+\beta_2)\, e_2, & [e_2,x] & =\beta_2\, e_2+\beta_n\, e_n, &&\\
[x,e_1]& =e_1+\gamma \, e_2,   & [e_i,x] &=-(i-1)\, e_i,  && 3\leq i\leq n,  \\
 [x,e_2]&=\delta\, e_2, & [x,x]& =\lambda_2\, e_2+\gamma\beta_n\, e_{n-1} \,. &&
\end{aligned}\right.\]

Considering the Leibniz identity for the elements $\{x, x, e_2\},
\ \{x, x, x\}, \ \{x, e_1, x\}$, we obtain:
\[\delta\beta_n=\delta(\delta+\beta_2)=\delta\lambda_2=\gamma\delta(\delta+\beta_2)=0.\]

Notice that if $e_2\in \Ann_r(R)$, then $\dim \Ann_r(R)=n-1$ and if
$e_2 \notin \Ann_r(R)$, then $\dim \Ann_r(R)=n-2$.

Now we analyze the following possible subcases:

\textbf{Case 1.1.} Let $e_2\in \Ann_r(R)$. Then $\delta=0$ and making the change
$e_1'=e_1+\gamma e_2$ we can assume that $[x,e_1]=e_1$.

In this case, we must consider two new subcases:

\textbf{Case 1.1.1.} Let $e_2\in \cent(R)$. Then $\dim
\cent(R)=1$ and $\beta_2=\beta_n=0$.  Then we have two options:
if $\lambda_2=0$, then we get the split algebra $R_1(0)$; if
$\lambda_2\neq 0$, then  we obtain the algebra $R_1(1)$ by scaling
the basis.

\textbf{Case 1.1.2.} Let $e_2 \notin \cent(R)$. Then $\dim
\cent(R)=0$ and $(\beta_2,\beta_n)\neq(0,0)$.

Let us take the following general change of basis:
\[e_1'=\sum\limits_{i=1}^nA_i\, e_i, \ \  e_2'=\sum\limits_{i=1}^nB_i\, e_i, \ \
e_i'=A_1^{i-2}\Big(A_1e_i+\sum\limits_{j=i+1}^nA_{j-i+2}\, e_i\Big), \
3\leq i\leq n, \ \ x'=\sum\limits_{i=1}^n C_i\, e_i+C_{n+1}\, x,\]
where $(A_1B_2-A_2B_1)C_{n+1}\neq0$.

From $0=[e_2',e_1']=[e_2',e_2']$, we obtain that $B_1=0,\ B_i=0,\
3\leq i\leq n-1$, i.e. $e_2'=B_2\, e_2+B_n\, e_n$ and $A_1B_2\neq 0$.

The equalities
\[e_1'=[x',e_1']=A_1 C_1\, e_3+\sum\limits_{i=4}^n A_1 C_{i-1}\, e_i+A_1 C_{n+1}\, e_1,\]
imply that
\[C_{n+1}=1,\qquad A_2=0, \qquad  A_3=A_1C_1, \qquad A_i=A_1C_{i-1}, \quad
4\leq i\leq n.\]

Similarly, from
\[B_2\beta_2'\, e_2+(B_n\beta_2'+\beta_n'A_1^{n-1})\, e_n=
\beta_2'\, e_2'+\beta_n'\, e_n'=[e_2',x']=B_2\beta_2 \,e_2
+\big(B_2\beta_n-(n-1)B_n\big)\, e_n,\] and
\begin{multline*}
\lambda_2'B_2\, e_2+\lambda_2'B_n\, e_n=\lambda_2'\, e_2'=[x',x'] =
(\lambda_2+C_2\beta_2)\, e_2 + (C_1^2-2C_3)\, e_3 \\
 {} + {} \sum\limits_{i=4}^{n-1}\big(C_1C_{i-1}-(i-1)C_i\big)\, e_i+
\big(C_1C_{n-1}-(n-1)C_n+C_2\beta_n\big)\, e_n
\end{multline*}
we obtain
$C_i=\dfrac{1}{(i-1)!}C_1^{i-1}, \ 3\leq i\leq n-1$ and
\[\beta_2'=\beta_2, \quad \beta_n'=\frac{B_2\beta_n-B_n(\beta_2+n-1)}{A_1^{n-1}}, \quad
\lambda_2'=\frac{\lambda_2+\beta_2C_2}{B_2}, \quad \lambda_2'B_n=C_1C_{n-1}-(n-1)C_n+C_2\beta_n.\]
Now we must distinguish two subcases:

 \textbf{Case 1.1.2.1.} Let
$\beta_2=1-n$. We put $C_2=-\frac{\lambda_2}{1-n}, \
C_n=\frac{C_1C_{n-1}+C_2\beta_n}{n-1}$, then we get $\lambda_2'=0$
and $\beta_n'=\frac{B_2\beta_n}{A_1^{n-1}}$.

If $\beta_n=0$, then we get the algebra $R_2(\alpha)$ for $\alpha=1-n$.

If $\beta_n\neq0$, then making $A_1=\sqrt[\uproot{2} n-1]{\beta_nB_2}$, we
obtain $\beta_n'=1$  and the algebra $R_3$.

\textbf{Case 1.1.2.2.} Let  $\beta_2\neq1-n$. Taking the change
$B_n=\frac{B_2\beta_n}{\beta_2+n-1}$, we obtain $\beta_n=0$. Since
$\beta_2\neq 0$, we set $C_2=-\frac{\lambda_2}{\beta_2}, \
C_n=\frac{C_1C_{n-1}+C_2\beta_n}{n-1}$ and we get $\lambda_2=0$,
i.e., the algebra $R_2(\alpha)$ is
obtained, for $\alpha\notin\{1-n,0\}$.

\textbf{Case 1.2.} Let  $e_2 \notin \Ann_r(R)$. Then
$\delta\neq0$ and $\beta_2=-\delta, \beta_n=\lambda_2=0$.

Let us consider the general change of basis in the following form:
\[e_1'=\sum\limits_{i=1}^nA_i\, e_i, \quad e_2'=\sum\limits_{i=1}^nB_i\, e_i, \quad
e_i'=A_1^{i-2}\big(A_1\, e_i+\sum\limits_{j=i+1}^nA_{j-i+2}\, e_j\big), \
3\leq i\leq n, \quad x'=\sum\limits_{i=1}^nC_i\, e_i+C_{n+1}\, x,\]
where $(A_1B_2-A_2B_1)C_{n+1}\neq0$.

Then from $0=[e_2',e_1']=[e_2',e_2']$, we derive that $B_1=0,\
B_i=0, \ 3\leq i\leq n-1$, i.e. $e_2'=B_2\, e_2+B_n\, e_n$ and
$A_1B_2\neq0$.

Similarly, from the equations:
\[e_1'+\gamma'\, e_2'=[x',e_1']=A_1C_{n+1}\, e_1+
C_{n+1}(A_1\gamma+A_2\delta)\, e_2+A_1C_1\,e_3
+\sum\limits_{i=4}^nA_1C_{i-1}\, e_i\] and
\[\delta'(B_2\, e_2+B_n\, e_n)=\delta'\, e_2'=[x',e_2']=B_2\delta\, e_2\]
we obtain
 \begin{align*}
 C_{n+1}&=1, & A_3 &=A_1C_1, &  A_i&=A_1C_{i-1}, \quad 4\leq i\leq n-1,  \\
\gamma'&=\frac{A_1\gamma+A_2(\delta-1)}{B_2}, & A_1C_{n-1}&=A_n+\gamma'B_n, & \delta'& =\delta, \qquad \qquad \delta'B_n =0\,.
 \end{align*}
Now we distinguish the following two subcases:

\textbf{Case 1.2.1.} Let  $\delta\neq 1$. Then by the
substitution  $A_2=-\dfrac{A_1\gamma}{\delta-1}, \ A_n=A_1C_{n-1}$
into the above conditions, we get $\gamma'=0$ and the algebra
$R_4(\alpha)$.

\textbf{Case 1.2.2.} Let $\delta=1$. Then $B_n=0$. In the case of
$\gamma=0$, we get $\gamma'=0$. In the case $\gamma\neq0$, by
putting $B_2=A_1\gamma$ and $A_n=A_1C_{n-1}-B_n$, we get
$\gamma'=1$. Thus, the algebras $R_5(\alpha), \ \alpha\in\{0,1\}$,
are obtained.

\textbf{Case 2.} Let $\alpha_1=0$. Then $\beta_2\neq0$ and by
replacing  $x$ by $x'=\dfrac{1}{\beta_2}\, x$, we can assume
$[e_2,x']=e_2+\beta_ne_n$.

Under these conditions, the table of multiplication of the
solvable algebra $R$ has the form:
\[\left\{\begin{aligned}
{} [e_1,x]& =\sum\limits_{i=2}^n\alpha_i\, e_i, & [e_2,x] & =e_2+\beta_n\, e_n, &&\\
[x,e_1]& =\sum\limits_{i=2}^n\gamma_i\, e_i,   & [e_i,x] &=\sum\limits_{j=i+1}^n\alpha_{j-i+2}\, e_j, && \ 3\leq i\leq n-1, \\
 [x,e_2]&=\delta\, e_2, & [x,x]& =\sum\limits_{i=2}^n\lambda_i\, e_i \,. &&
\end{aligned}\right.\]

Making the transformation $x'=x-\gamma_3\, e_1-\sum\limits_{i=3}^{n-1}\gamma_{i+1}\, e_i$, we can assume that
$[x,e_1]=\gamma\, e_2$.

Similarly as above, we obtain  the conditions:
\[\gamma(\delta+1)=\alpha_2\delta-\gamma=\beta_n\delta=\delta(\delta+1)=\lambda_2\delta=0.\]

Now we distinguish the following subcases depending on the
possible values of the parameter $\delta$:

\textbf{Case 2.1.} Let  $\delta\neq 0$. Then $\dim \Ann_r(R)=n-2$
and $\beta_n=\lambda_2=0,\ \delta=-1, \ \alpha_2=-\gamma$. By
means of the  change of the  basic element $e_1'=e_1+\gamma e_2$,
we can suppose that $[x',e_1]=0$.

Taking the general change of bases as in the above considered
cases, we derive the conditions for the parameters under the
following basis transformation:
\[\alpha_i'=\frac{\alpha_i}{A_1^{i-2}}, \quad  3\leq i\leq n, \quad \quad \lambda_n'=\frac{\lambda_n}{A_1^{n-1}}.\]
Consequently, we deduce the algebra $R_6(\alpha_3,\alpha_4,\dots,\alpha_n,\lambda, -1)$.

\textbf{Case 2.2.} Let  $\delta=0$. Then $\dim \Ann_r(R)=n-1$ and
$\gamma=0$. Taking the change of basis $e_2'=e_2+\beta_n\, e_n$, we
can assume that $[e_2,x]=e_2$ and by the change $x'=x-\lambda_2\,e_2$,
 we  can also suppose that $[x,x]=\lambda_n\, e_n$. Therefore,
we have the products
\[[e_1,x]=\sum\limits_{i=2}^n\alpha_i\, e_i, \quad [e_2,x]=e_2, \quad
[e_i,x]=\sum\limits_{j=i+1}^n\alpha_{j-i+2}\, e_j, \ 3\leq i\leq
n-1, \qquad [x,x]=\lambda_n\, e_n.\]
 Applying similar arguments to general transformation of
 bases, we have
\[\alpha_2'=0, \qquad \alpha_i'=\frac{\alpha_i}{A_1^{i-2}}, \ \  3\leq i\leq n, \qquad \lambda_n'=\frac{\lambda_n}{A_1^{n-1}}.\]
Thus, we obtain the algebra $R_6(\alpha_3,\alpha_4,\dots,\alpha_n,\lambda, 0)$.
\end{proof}

\begin{thm} \label{thm5.8} An arbitrary $(n+2)$-dimensional solvable Leibniz algebra with nilradical $F_n^2$
 is isomorphic to one of the following non isomorphic algebras:
\[ L_1: \left\{\begin{aligned}
{} [e_1,e_1] & =e_3, & [e_i,e_1] & =e_{i+1}, && 3\leq i\leq n-1,\\
 [e_1,x]& =e_1, & [x,e_1] & =-e_1,   && \\
 [e_2,y]& =-[y,e_2]=e_2, &[e_i,x]&=(i-1)e_i, && 3\leq i\leq n,
\end{aligned}\right.\]
\[ L_2: \left\{\begin{aligned}
{} [e_1,e_1] & =e_3, & [e_i,e_1] & =e_{i+1}, && 3\leq i\leq n-1,\\
 [e_1,x]& =e_1, & [x,e_1] & =-e_1,   && \\
 [e_2,y]& =e_2, &[e_i,x]&=(i-1)e_i, && 3\leq i\leq n\,.
\end{aligned}\right.\]
\end{thm}
\begin{proof} Let  \[
\mathcal{R}_{x _{\mid F_n^2}}=\begin{pmatrix}
\alpha_1& \alpha_2&\alpha_3&\alpha_4&\dots&\alpha_{n-1}&\alpha_n\\
0& \beta&0&0&\dots&0& \gamma\\
0& 0&2\alpha_1&\alpha_3&\dots&\alpha_{n-2}& \alpha_{n-1}\\
0& 0&0&3\alpha_1&\dots&\alpha_{n-3}&\alpha_{n-2}\\
\vdots&\vdots&\vdots&\vdots&\dots&\vdots&\vdots\\
0&0&0&0&\dots&0&(n-1)\alpha_1
\end{pmatrix}
\]
 and  \[
\mathcal{R}_{y _{\mid F_n^2}}=\begin{pmatrix}
\lambda_1& \lambda_2&\lambda_3&\lambda_4&\dots&\lambda_{n-1}&\lambda_n\\
0& \mu &0&0&\dots&0& \nu\\
0& 0&2\lambda_1&\lambda_3&\dots&\lambda_{n-2}& \lambda_{n-1}\\
0& 0&0&3\lambda_1&\dots&\lambda_{n-3}&\lambda_{n-2}\\
\vdots&\vdots&\vdots&\vdots&\dots&\vdots &\vdots\\
0&0&0&0&\dots&0&(n-1)\lambda_1
\end{pmatrix}
\]
be two nil independent outer derivation of the algebra $F_n^2$.

Taking the change of the basic elements $x, y$ similar to \eqref{E:chbasis},
we can assume that $\alpha_1=\mu=1, \ \lambda_1=\beta=0$.

Thus, we have the products:
\begin{align*}
 [e_1,x]&=e_1+\sum\limits_{i=2}^n\alpha_i\, e_i,
& [e_2,x]& =\gamma \, e_n, & [e_i,x] & =(i-1)\,e_i
+\sum\limits_{j=i+1}^n\alpha_{j-i+2}\, e_j,  && \ 3\leq i\leq n,\\
[e_1,y]& =\sum\limits_{i=2}^n\lambda_i\, e_i, & [e_2,y] & =e_2+\nu \,e_n,
& [e_i,y] & =\sum\limits_{j=i+1}^n\lambda_{j-i+2}\, e_j, && \ 3\leq i\leq n.
\end{align*}

Applying similar reasonings and changes of bases which we have
used in Theorem \ref{thm5.7}, we obtain isomorphism classes of algebras whose representative elements are the  $L_1$ and $L_2$.
\end{proof}

\begin{rem} In fact, the algebra $L_1$ is a direct sum of the ideals $(NF_{n-1}+ \langle x \rangle)$ and $\langle e_2, y \rangle$,
 where the  sum $NF_{n-1}+ \langle x \rangle$ is a solvable Leibniz algebra with nilradical $NF_{n-1}$ and $\langle e_2, y \rangle$ is
 a two-dimensional solvable Lie algebra. The algebra $L_2$ is a direct sum of the ideals $(NF_{n-1}+ \langle x \rangle)$ and
 $\langle e_2, y \rangle$, where $\langle e_2, y \rangle$ is a two-dimensional solvable non-Lie Leibniz algebra. Thus, from Theorem
 \ref{thm5.8},
 we conclude that any $(n+2)$-dimensional solvable Leibniz algebra with nilradical $F_n^2$ is split.
\end{rem}

\section*{Acknowledgements} The two first authors were supported by
MICINN, grant MTM 2009-14464-C02 (Spain) (European FEDER support
included), and by Xunta de Galicia grant Incite09 207 215PR.



\end{document}